\documentclass[preprint,1p]{elsarticle}

\makeatletter
 \def\ps@pprintTitle{%
 	\let\@oddhead\@empty
 	\let\@evenhead\@empty
 	\def\@oddfoot{\footnotesize\itshape
 		{} \hfill\today}%
 	\let\@evenfoot\@oddfoot
 }
\makeatother
\usepackage[unicode]{hyperref}

\usepackage{latexsym}
\usepackage{indentfirst}
\usepackage{amsxtra}
\usepackage{amssymb}
\usepackage{amsthm}
\usepackage{amsmath}
\usepackage{amsfonts}
\usepackage{mathrsfs} 

\usepackage{xcolor}
\usepackage{color}

\usepackage{amsfonts}
\usepackage{mathtools}

\usepackage[capitalise]{cleveref}

\newtheorem{theor}{Theorem}[section]
\newtheorem{prop}[theor]{Proposition}
\newtheorem{cor}[theor]{Corollary}
\newtheorem{lemma}[theor]{Lemma}
\theoremstyle{definition} 
\newtheorem{defin}[theor]{Definition}
\newtheorem{rem}[theor]{Remark}
\newtheorem{conj}{Conjecture}
\newtheorem{ex}{Example}

\newtheorem{conv}{Convention}

\DeclareMathOperator{\Sym}{Sym}
\DeclareMathOperator{\id}{id}

\DeclareMathOperator{\Ret}{Ret}

\begin{document}

\begin{frontmatter}
	\title{On uniconnected solutions of the Yang-Baxter equation and Dehornoy's class
}
	\author[*]{M.~CASTELLI	\tnoteref{mytitlenote1}}
		\ead{marco.castelli@unisalento.it - marcolmc88@gmail.com}
	\tnotetext[mytitlenote1]{The author is member of GNSAGA (INdAM).}
 
	\author{S. RAMIREZ }
	\ead{sramirez@dm.uba.ar}

\begin{abstract}
In the first part, we focus on indecomposable involutive solutions of the Yang-Baxter equation whose permutation group forces them to be uniconnected. Indecomposable involutive solutions with a permutation group isomorphic to a dihedral group or a minimal non-cyclic group are studied in detail. In the last part, we study the Dehornoy's class of involutive solutions (not necessarily indecomposable) and its link with left braces. As an application, we give an upper bound for several families of indecomposable involutive solutions and we compute the precise value in some other cases.
\end{abstract}

\begin{keyword}
\texttt{set-theoretic solution\sep Yang-Baxter equation\sep brace \sep indecomposable solution}
\MSC[2020] 
16T25\sep 81R50 
\end{keyword}
\end{frontmatter}

\section*{Introduction}

The quantum Yang-Baxter equation comes from theoretical physics and it first appeared in a paper by C.N. Yang \cite{yang1967}. In 1992 Drinfel'd \cite{drinfeld1992some} suggested the study of the set-theoretical version of this equation. Specifically, a pair $(X,r)$ is said to be a  \emph{set-theoretic solution of the Yang-Baxter equation} if $X$ is a non-empty set, and $r:X\times X\to X\times X$ is a map such that the relation
\begin{align*}
\left(r\times\id_X\right)
\left(\id_X\times r\right)
\left(r\times\id_X\right)
= 
\left(\id_X\times r\right)
\left(r\times\id_X\right)
\left(\id_X\times r\right)
\end{align*}
is satisfied.  
Writing a solution $(X,r)$ as $r\left(x,y\right) = \left(\sigma_x\left(y\right),\tau_y\left(x\right)\right)$, with
$\sigma_x, \tau_x$ maps from $X$ into itself, for every $x\in X$, we say that $(X, r)$ is \emph{non-degenerate} if $\sigma_x,\tau_x\in \Sym_X$, for every $x\in X$.\\
After the papers of Gateva-Ivanova and Van Den Bergh  \cite{gateva1998semigroups} and Etingov, Schedler, and Soloviev \cite{etingof1998set}, the attention of several people focused on the \emph{involutive} non-degenerate set-theoretic solutions, where a set-theoretic solution $(X,r)$ is said to be involutive if $r^2$ is the identity map on $X\times X$. From now on, even if not specified, every involutive non-degenerate set-theoretic solution of the Yang-Baxter equation will be simply called \emph{involutive solution}.\\
Among involutive solutions, in \cite[Section 2]{etingof1998set} the class of \textit{indecomposable solutions} was introduced.  The interest in these solutions is motivated by the fact that they are the basic solutions that allow to construct and classify all the other ones, not necessarily indecomposable, by suitable construction-tools, such as twisted union, dynamical extension and retraction (see \cite{cacsp2018,etingof1998set,vendramin2016extensions} for more details). However, the classification-problem of involutive solutions is far from solved even if we restrict to indecomposable ones. In this context, several machineries were developed. In particular, several people considered the structure of left brace on the permutation group $\mathcal{G}(X,r)$ of an involutive solution $(X,r)$. Remarkable results obtained in  \cite{bachiller2016solutions,rump2020} states that every indecomposable involutive solution $(X,r)$ having permutation left brace $\mathcal{G}(X,r)$ isomorphic to a left brace $B$ can be obtained by means of a suitable core-free subgroup $H$ of the multiplicative group of $B$ and considering the left cosets of $B$ by $H$ as underlying set. Moreover, by \cite[Theorem 1]{rump2020}, every indecomposable involutive solution $(X,r)$ is an epimorphic image of an indecomposable involutive solution $(Y,s)$ such that $\mathcal{G}(Y,s)$ acts regularly on the set $Y$. Therefore, the classification of indecomposable solutions $(X,r)$ such that $\mathcal{G}(X,r)$ acts regularly on $X$, which in \cite[Definition 1]{rump2020} are named \emph{uniconnected}, is the first milestone for the classification of all the indecomposable involutive solutions. A natural approach consists in the study of the uniconnected ones for which the permutation group $\mathcal{G}(X,r)$ belong to a specific class. Indecomposable solutions with abelian permutation group were recently studied (of course, these solutions always are uniconnected) and very strong results were obtained. Indeed, in \cite{JePiZa20x} the ones with multipermutation level $2$ were completely classified, while in \cite{jedlivcka2021cocyclic,rump2020classifi} the restriction on the multipermutation level was removed for the ones with cyclic permutation group. Further theoretical results on these solutions are recently showed in \cite{castelli2023studying}, where the infinite case also is considered. Much less is known about uniconnected solutions with non-abelian permutation group. In \cite{castelli2021classification} a description of the ones with a Z-group permutation group of odd order were given, giving a complete classification if the underlying set has square-free odd order, while in \cite[Section 4]{rump2022class} uniconnected solutions provided by primary cyclic left braces (i.e. left braces having prime-power order and cyclic additive group) were completely classified.  Indecomposable involutive solutions not necessarily uniconnected are recently studied in \cite{cedo2022indecomposable,dietzel2023indecomposable}, where the ones having square-free order are considered, and in \cite{JePiZa22x}, where the ones with multipermutation level $2$ are characterised. \\
In this paper, we study indecomposable involutive solutions $(X,r)$ for which the structure of the permutation group $\mathcal{G}(X,r)$ force them to be uniconnected. After recovering basic definition and results in the first section, in Section $2$ we consider indecomposable solutions provided by cyclic left braces and we show that these solutions and their retractions always are uniconnected. As an application, we show that indecomposable solutions with dihedral permutation group of size $2n$ (where $n$ is an odd number) and indecomposable solutions with square-free permutation group always are uniconnected. They are completely classified in Section $3$. In Section $4$, we turn our attention to indecomposable solutions for which the permutation group is such that all its non-trivial subgroups are cyclic. Since the ones with cyclic permutation group are completely classified (see \cite{jedlivcka2021cocyclic}), we focus to the non-cyclic case. These groups are completely classified in \cite{miller1903non} and we will call them  \emph{minimal non-cyclic}. Apart two exceptions, we have to consider non-abelian groups having size $q^n p$ (where $q$ and $p$ are distinct prime numbers) and with a normal subgroup of size $p$. We show that indecomposable solutions with minimal non-cyclic group are uniconnected and, as a final result of the section, a quite explicit classification is given.\\
In the last part of the paper, we study the Dehornoy's class of a (not necessarily indecomposable) solution, an important invariant related to Garside monoids and groups (see \cite{chouraqui2010garside,dehornoy2015set} for more details). Following the approach of \cite{lebed2022involutive}, where the Dehornoy's class is related to left braces, we show that, given an involutive solution $(X,r)$, its Dehornoy's class concides with the maximum order of an element in the additive group of the permutation left brace $\mathcal{G}(X,r)$. We use this fact to provide an upper bound of the Dehornoy's class to several families of involutive solutions, giving an evidence to \cite[Conjecture $3.6$]{feingesicht2023dehornoy}. Finally, we specialize our results to indecomposable involutive solutions, computing the Dehornoy's class of some families of indecomposable involutive solutions and providing a positive answer to \cite[Conjecture $3.4$]{feingesicht2023dehornoy} in several cases.

\section{Basic definitions and results}

We collect in this section some basic definitions and results on left braces and set-theoretical solutions to the Yang-Baxter equation.

\begin{defin}
A \emph{set-theoretical solution to the Yang-Baxter equation} is a pair $(X,r)$, with $X$
a set and $r : X\times X \to X\times X$ a function that satisfies
\[
(r\times\id)(\id\times r)(r\times\id) = (\id\times r)(r\times\id)(\id\times r).
\]
\end{defin}

\noindent We will restrict to the class of \emph{non-degenerate} and \emph{involutive} set-theoretical solutions.

\begin{defin}
Let $(X,r)$ be a set-theoretical solution to the Yang-Baxter equation and write
\(r(x,y)=(\sigma_x(y),\tau_y(x))\). We say the solution is \emph{non-degenerate}
if all the maps \(\sigma_x,\tau_y:X\to X\) are bijections. We say the solution
is \emph{involutive} if \(r^2=\id\). 
\end{defin}

\noindent As for other algebraic objects, we can define the notion of isomorphism. Specifically, two set-theoretical solutions to the Yang-Baxter equation $(X,r)$ and $(X',r')$ are \emph{isomorphic} if there exist a bijective map $f$ from $X$ to $X'$ (that will be called \emph{isomorphism}) such that $r' (f\times f)=(f\times f)r$.

\begin{conv}
    From now on, even if not specified, every set-theoretical solution to the Yang Baxter equation, which we will simply call \emph{solution}, will be involutive and non-degenerate.
\end{conv} 

\noindent In \cite{etingof1998set}, Etingof, Schedler and Soloviev introduced the \textit{retract relation} for involutive non-degenerate solutions, an equivalence relation which we denote by $\sim_r$. If $(X,r)$ is an involutive non-degenerate solution, then $x\sim_r y$ if and only if $\sigma_x=\sigma_y$, for all $x,y\in X$. In this way, we can define canonically another solution, having the quotient $X/\sim_r$ as underlying set, which is named \emph{retraction} of $(X,r)$ and is indicated by $\Ret(X, r)$. 
Of course, one can iterate the retraction-process obtaining the $n$-th retraction, for an arbitrary natural number $n$, which we indicate by $\Ret^n(X, r)$. 

\begin{defin}
  An involutive solution $(X,r)$ is said to be a \emph{multipermutation solution of level $n$} if $n$ is the minimal non-negative integer such that $\Ret^n(X)$ has cardinality one.  
\end{defin}

By a standard exercise, one can verify that the multipermutation level of an involutive solution is an isomorphisms invariant. An other important invariant, studied in several papers \cite{lebed2022involutive}, is the Dehornoy's class, that we define below. At first, recall that if $(X,r)$ is an involutive solution and $\cdotp$ is the binary operation on $X$ given by $x\cdotp y:=\sigma_x^{-1}(y)$ for all $x,y\in X$, the map $\Omega_n$ from $X^n$ to $X$ (where $X^n$ is the cartesian product of $X$ $n$ times) is inductively given by  $$\Omega_1(x_1)=x_1\hspace{2mm} and \hspace{2mm}\Omega_n(x_1,...,x_n):=\Omega_{n-1}(x_1,...,x_{n-1})\cdotp \Omega_{n-1}(x_1,...,x_{n-2},x_n)$$
    for all $n\in \mathbb{N}$, $n>1$ and $x_1,...,x_n\in X$.

\begin{defin}
    Let $(X,r)$ be an involutive solution. Then, $(X,r)$ has \emph{Dehornoy's class $n$} if $n$ is the minimal non-negative integer such that $\Omega_{n+1}(x,x,...,x,y)=y$ for all $x,y\in X$.
\end{defin}

\noindent Given an involutive solution $(X,r)$, we
define its \emph{permutation group} $\mathcal{G}(X,r)$ as the group of permutations
on $X$ generated by the maps $\sigma_x$. Using the permutation group we can define \emph{indecomposable} solutions as follows (our definition is equivalent to the original one given in \cite{etingof1998set}).

\begin{defin}
We say that $(X,r)$ is \emph{indecomposable}
if its permutation group $\mathcal{G}(X,r)$ acts transitively on $X$. If furthermore $\mathcal{G}(X,r)$ acts regularly on $X$, then $(X,r)$ will be called \emph{uniconnected}.
\end{defin}

\noindent Besides, the permutation group $\mathcal{G}(X,r)$ of an involutive solution can be
endowed with an additional group structure that always is abelian. This two group operations
make $\mathcal{G}(X,r)$ into an algebraic structure called \emph{left brace}.

\begin{defin}[\cite{cedo2014braces}, Definition 1]
A set $B$ endowed of two operations $+$ and $\circ$ is said to be a \textit{left brace} if $(B,+)$ is an abelian group, $(B,\circ)$ a group, and the equality $a\circ (b + c)= a\circ b - a + a\circ c$ follows, for all $a,b,c\in B$. 
\end{defin}

\noindent From now on, if $(B,+,\circ)$ is a left brace, the group $(B,+)$ will be called the \textit{additive group} and the group $(B,\circ)$ will be called the \textit{multiplicative group}. The most elementary example of left brace can be constructed taking and abelian group $(B,+)$ and setting $a\circ b:=a+b$ for all $a,b\in A$. Through the paper, we call these left braces \emph{trivial}. 

\begin{ex}
    Let $p$ be a prime number, $n$ a natural number and $t\in \{1,...,n\}$. Set $(B,+):=(\mathbb{Z}/p^n\mathbb{Z},+)$. Then, the operation $\circ$ given by $a\circ b:=a+b+p^t ab$ makes $(B,+)$ into a left brace that we will denote by $C(p,n,t)$. Moreover, by \cite[Proposition $3$]{rump2007braces} every left brace with cyclic multiplicative group is isomorphic to the direct product of such a left braces.
    \end{ex}

\begin{conv}
Given a left brace $(B,+,\circ)$ and an element $x$ of $B$, we will indicate by $o(x)_+$ the order of $x$ in $(B,+)$ and by $o(x)_\circ$ the order of $x$ in $(B,\circ)$.
\end{conv}

As we said before, every involutive solution give rise to a left brace on the permutation group $\mathcal{G}(X,r)$, where the multiplicative operation coincides with the usual composition. We will refer to $(\mathcal{G}(X,r),+,\circ)$ as the \emph{permutation left brace}.
The multiplicative group of a left brace acts by group automorphisms on the additive
group. The action is given by
\begin{align*}
  \lambda_{-} : (B,\circ) &\to \operatorname{Aut}(B,+)\\
  a &\mapsto (\lambda_a : x \mapsto -a + a \circ x)
\end{align*}

\noindent Similarly to rings, the notion of ideal of a left brace was given in \cite{rump2007braces} and reformulated in \cite[Definition 3]{cedo2014braces}.

\begin{defin}
Let $B$ be a left brace. A subset $I$ of $B$ is said to be a \textit{left ideal} if it is a subgroup of the multiplicative group and $\lambda_a(I)\subseteq I$, for every $a\in B$. Moreover, a left ideal is an \textit{ideal} if it is a normal subgroup of the multiplicative group.
\end{defin}

\noindent If $I$ is a left ideal, then $(I,\circ)$ always is a subgroup of $(B,\circ)$, even if in general is not a normal subgroup. For example, if $B$ is a left brace and $(B_p,+)$ is a Sylow subgroup of the additive group $(B,+)$, then $(B_p,+)$ is a left ideal of $B$. If in addition $ (B,\circ)$ is abelian, then $B_p$ is an ideal.\\
If $I$ is an ideal of a left brace $B$, then the structure $B/I$ is a left brace called the \emph{quotient left brace} of $B$ by $I$.  A special ideal of a left brace, introduced in \cite{rump2007braces}, is the socle. 

\begin{defin}
Let $B$ be a left brace. Then, the set 
$$
    Soc(B) := \{a\in B \ | \ \forall \,     b\in B \quad  a + b = a\circ b \}
$$
is called \emph{socle} of $B$.
\end{defin}


\noindent A left brace $B$ always provide an involutive non-degenerate solution $(B,r)$ by $\sigma_a(b):=\lambda_a(b)$ and $\tau_b(a):=\lambda^{-1}_{\lambda_a(b)}(a)$ for all $a,b\in B$. Moreover, we have that the underlying set of $\Ret(B,r)$ can be identified with the set $B/Soc(B)$. From now on, we will say that $B$ has \emph{multipermutation level $n$} if the canonical solution $(B,r)$ provided by $B$ is a multipermutation solution having level $n$.

\medskip

A subset $X$ of a left brace $B$ is called a \emph{cycle base} if it is invariant
under the $\lambda$-action and it generates the additive group of the left brace. If
moreover $B$ acts transitively on $X$ we say that it is a \emph{transitive cycle base}. Left braces and cycle bases are the key-ingredients useful to construct all the involutive solutions.

\begin{theor}[Theorem 3.1, \cite{bachiller2016solutions}]\label{costruz}
    Let $B$ be a left brace, $\{B_i\}_{i\in K}$ the set of the orbits of $B$ respect to the $\lambda$ action and, for every $i\in K$ let $a_i\in B_i$. Moreover, let $I$ be a subset of $K$ such that $Y=\bigcup_{i\in I} B_i$ is a cycle base and, for each $i\in I$, let $\{K_{i,j} \}_{j\in J_i}$ a family of subgroups of $St(a_i)$ such that $\bigcap_{i\in I,j\in J_i} core(K_{i,j})=\{ id_B\}$. Then, the pair $(X,r)$ given by $X:=\bigcup_{i\in I,j\in J_i} B/K_{i,j}$ and $\sigma_{x\circ K_{i_x,j_x}}(y\circ K_{i_y,j_y}):=\lambda_x(a_{i_x})\circ y \circ K_{i_y,j_y}$ give rise to an involutive solution such that $\mathcal{G}(X,r) \cong (B,\circ) $.\\
    Conversely, every involutive solutions $(X,r)$ with $\mathcal{G}(X,r) \cong B $ (as left braces) can be obtained in this way.
\end{theor}

\begin{rem}\label{remark1}
By the previous theorem, using the same notation, every solutions having $B$ as permutation left brace can be constructed choosing $I\subseteq K$ such that $<\{B_i\}_{i\in I}>_+=B $, $\mathcal{A}=\{a_i\}_{a_i\in B_i}$ and $\mathcal{K}=\{\{K_{i,j} \}_{j\in J_i}\}_{i\in I}$ a family of subgroups such that $K_{i,j}\subseteq St(a_i)$ and $\bigcap_{i\in I,j\in J_i} core(K_{i,j})=\{ id_B\}$. We indicate this solution by $(X_{B,I,\mathcal{A},\mathcal{K}},r)$. We remark that difference choices of $ I,\mathcal{A},\mathcal{K}$ could give isomorphic solutions: however, if $a_i$ and $a'_i$ are two elements of $B_i$, then they have the same additive order.
\end{rem}

\noindent If $(X,r)$ is an indecomposable involutive solution, \cref{costruz} can be simplified. 

\begin{prop}[Theorem 3, \cite{rump2020}]\label{cosmod}
Let $(B,+,\circ)$ be a left brace, $Y\subset B$ a transitive cycle base, $a_1\in Y$,
and $K \subset St(a_1)$ a core-free subgroup of $(B,\circ)$, contained in the
stabilizer of $a_1$. Then, the pair $(X,r)$ given by $X:=B/K$ and $\sigma_{x\circ K}(y\circ K):=\lambda_x(a_{1})\circ y \circ K$ give rise to an indecomposable involutive solution with $ \mathcal{G}(X,r)\cong (B,\circ)$.\\
 Conversely, every indecomposable involutive solutions $(X,r)$ with $\mathcal{G}(X,r) \cong B $ (as left braces) can be obtained in this way.
\end{prop}

\noindent In \cite{bachiller2016solutions} a tool to cut the isomorphic solutions was provided. For our scope, we report this result for indecomposable solutions. At first, recall that if $B$ is a left brace, a bijective map $\psi$ from $B$ to itself is said to be an \emph{automorphism} of left brace if $\psi\in Aut(B,+)\cap Aut(B,\circ)$.

\begin{prop}[Theorem 4.1, \cite{bachiller2016solutions}]\label{bachi}
Let $B$ be a left brace, $(X_1,r_1)$ and $(X_2,r_2)$ be two solutions constructed as before from $Y_1,a_1,K_1$ and
$Y_2,a_2,K_2$ respectively. Then $(X_1,r_1)$ and $(X_2,r_2)$ are isomorphic if and only if
there exist $z\in B$ and \(\psi\in Aut(B,+,\circ)\) such that \(\psi(a_1)=\lambda_z(a_2)\)
and \(\psi(K_1) = g\circ K_2\circ g^{-}\)
\end{prop}

\begin{rem}\label{remark11}
If $(X,r)$ is an uniconnected solution, then \cref{cosmod} and \cref{bachi} furtherly simplify since in this case the core-free subgroup $K$ must be trivial.
\end{rem}

\section{Uniconnection of some indecomposable involutive solutions}

The goal of this section is showing that the the permutation group of some indecomposable solutions force them to be uniconnected. As a main result, we will show that left braces with cyclic additive group only provide uniconnected solutions.

\smallskip

\noindent We start by the following easy result. Recall that a group is said to be a \emph{Dedekind group} if all its subgroups are normal.

\begin{prop}\label{solded}
Let $(X,r)$ be an indecomposable solution such that $\mathcal{G}(X,r)$ is a Dedekind group. Then, $(X,r)$ is an uniconnected solution.
\end{prop}

\begin{proof}
Clearly, $\mathcal{G}(X,r)$ can not have non-trivial core-free subgroup, hence the thesis follows by \cref{cosmod}.
\end{proof}

\noindent A left brace is said to be \emph{cyclic} if so is its additive group. This class of left braces are considered in several papers (see for example \cite{rump2007classification,rump2019classification}). 

\begin{theor}\label{solcyc}
Let $B$ be a finite cyclic left brace and $(X,r)$ be an indecomposable solution with $B$ as permutation left brace. Then, $(X,r)$ is an uniconnected solution.
\end{theor}

\begin{proof}
By \cref{cosmod}, $(X,r)$ can be constructed starting from a transitive cycle base $Y$, an element $a_1\in Y$ and a core-free subgroup $K$ of $B$ contained in $St(a_1)$. Moreover, by \cite[Lemma $3.2$]{castelli2021classification}, $a_1$ is an additive generator of $B$. Now, by \cite[Proposition $3.3$]{castelli2021classification} we have that $St(a_1)$ coincides with $Soc(B)$, therefore it is a normal subgroup of $(B,\circ)$ and, since $(Soc(B),\circ)=(Soc(B),+)$, we have that $(Soc(B),\circ)$ is cyclic. Moreover, being all the subgroups of $(Soc(B),\circ)$ characteristic in $Soc(B)$, we obtain that every multiplicative subgroup contained in $Soc(B) $ is a normal subgroup of $(B,\circ)$, hence $K$ must be trivial.
\end{proof}

\noindent In addition, cyclic left braces provide further solutions that are uniconnected by retractions.

\begin{cor}
    Let $B$ be a finite cyclic left brace and $(X,r)$ be an indecomposable solution with $B$ as permutation left brace. Then, $\Ret^n(X,r)$ is an uniconnected solution, for all $n\in \mathbb{N}\cup \{0\}$.
\end{cor}

\begin{proof}
    If $n=0$ the thesis follows by \cref{solcyc}. Since by \cite[Theorem 3.5]{castelli2021classification}, the permutation left brace of $\Ret(X,r)$ is the factor left brace $B/Soc(B)$, that clearly has cyclic additive group, the thesis follows by an easy induction on $n$.
\end{proof}

\noindent In general, the previous result can not be extended if $B$ has not cyclic additive group, as we can see in the following example.

\begin{ex}
Let $(X,r)$ be the involutive solution given by $X:=\{1,2,3,4,5,6,7,8\}$, $\sigma_1=\sigma_2:=(1,2)(3,5)(4,7)(6,8)$, $\sigma_3=\sigma_8:=(1,6,4,3)(2,5,7,8)$, $\sigma_4=\sigma_7:=(1,3,4,6)(2,8,7,5)$, $\sigma_5=\sigma_6:=(1,7)(2,4)(3,8)(5,6) $ and $\tau_y(x):=\sigma^{-1}_{\sigma_x(y)}(y)$ for all $x,y\in X$. Then, $(X,r)$ is an uniconnected solution, but $\Ret(X,r)$ is not uniconnected.
\end{ex}

\noindent In this way, indecomposable solutions with a Z-group permutation group are uniconnected, with one exception.

\begin{cor}\label{solZ}
Let $(X,r)$ be an indecomposable solution with a Z-group permutation group $\mathcal{G}(X,r)$ having no Sylow $2$-subgroups of size $4$. Then, $(X,r)$ is an uniconnected solution.
\end{cor}

\begin{proof}
By \cite[Corollary 2]{rump2019classification}, the permutation left brace of $(X,r)$ is a cyclic left brace. Then, the thesis follows by \cref{solcyc}.
\end{proof}

\noindent In \cite[Conjecture 6.9]{ramirez2022indecomposable} the second author conjectured that if $(X,r)$ is an indecomposable involutive solutions such that the permutation group $\mathcal{G}(X,r)$ is isomorphic to a dihedral group, then it is uniconnected. As an application of the previous results, we show the conjecture for several values of $n$.

\begin{cor}\label{uniconndihe}
Let $(X,r)$ be an indecomposable solution such that the permutation group $\mathcal{G}(X,r)$ is isomorphic to the dihedral group having order $2n$ for an odd number $n$. Then, $(X,r)$ is an uniconnected solution.
\end{cor}

\begin{proof}
    Since $n$ is an odd number, a dihedral group of size $2n$ is a Z-group with no Sylow subgroups having order $4$, hence the thesis follows by \cref{solZ}.
\end{proof}

\noindent In full generality Conjecture $6.9$ of \cite{ramirez2022indecomposable} is false. Indeed, in \cite[Example 1.2]{vendramin2016extensions} Vendramin exhibited an indecomposable solution having size $4$ with permutation group isomorphic to the dihedral group having $8$ elements.

\begin{ex}
Let $X:=\{1,2,3,4\}$ and $r$ be the solution given by $\sigma_1:=(3\;4)$, $\sigma_2:=(1\;3\;2\;4)$, $\sigma_3:=(1\;4\;2\;3)$, $\sigma_4:=(1\;2)$ and $\tau_y(x):=\sigma^{-1}_{\sigma_x(y)}(x)$, for all $x,y\in X$. Then, $(X,r)$ is an indecomposable solution having $4$ element such that $\mathcal{G}(X,r)$ is isomorphic to the dihedral group of size $8$.
\end{ex}

\noindent Inspecting all the involutive solutions of size $\leq 9$ by GAP \cite{Ve15pack}, in addition to the previous example we find a further indecomposable involutive solution having size $4$ and with permutation group isomorphic to the dihedral group having order $8$. These are the only counterexamples to \cite[Conjecture 6.9]{ramirez2022indecomposable} among the involutive solutions having size $\leq 9$.

\smallskip 

By a further inspection of involutive solutions of size $8$, we find an indecomposable involutive solution with permutation group isomorphic to a quaternion $2$-group. This solution is uniconnected. This is not a special case, as the following result show.


\begin{prop}\label{solqua}
Let $(X,r)$ be an indecomposable solution such that $\mathcal{G}(X,r)$ is isomorphic to a generalized quaternion $2$-group. Then, $(X,r)$ is an uniconnected solution.
\end{prop}

\begin{proof}
If $(X,r)$ is not uniconnected, there exist a non-trivial core-free subgroup of $\mathcal{G}(X,r)$. This implies that there exist an embedding of $\mathcal{G}(X,r)$ into $Sym(n)$, for some $n<|\mathcal{G}(X,r)|$, but this contradicts \cite[Theorem 1]{johnson1971minimal}.
\end{proof}

\noindent Even if \cite[Conjecture 6.9]{ramirez2022indecomposable} in full generality is false, in \cite[Theorem 6.1]{ramirez2022indecomposable} the second author showed that the possible sizes of an indecomposable involutive solution $(X,r)$ with a dihedral permutation group $\mathcal{G}(X,r)$ of size $2n$ are only $n$ and $2n$. The same result follows if $\mathcal{G}(X,r)$ is a \emph{generalized dihedral group}, where a group $H$ is said to be a generalized dihedral group if it is a semidirect product $A\rtimes \mathbb{Z}/2\mathbb{Z}$ for some abelian group $A$, where $\mathbb{Z}/2\mathbb{Z}$ acts on $A$ by inversion (of course, dihedral groups are generalized dihedral groups).

\begin{prop}\label{soldihe}
Let $(X,r)$ be an indecomposable solution such that $\mathcal{G}(X,r)$ is isomorphic to a generalized dihedral group of size $2n$. Then, $|X|=2n$ or $|X|=n$.
\end{prop}

\begin{proof}
It is the same idea of \cite[Theorem 6.1]{ramirez2022indecomposable}.
\end{proof}

\section{Some classification results}

In this section, we extend the results obtained in \cite{ramirez2022indecomposable} on indecomposable involutive solutions with permutation group having size $pq$ and $p^2q$ in two direction. In the first one, we classify indecomposable involutive solutions with dihedral permutation group and size $2n$, for an arbitrary odd number $n$. In the second one, we classify the indecomposable involutive solutions  having a permutation group of square-free size.

\smallskip

\noindent At first, we need a couple of lemma.

\begin{lemma}\label{autcyc}
Let $B$ be a cyclic left brace of size $p^k$, for some odd prime number $p$. Then, $B$ has an automorphism $\alpha$ of order $2$ if and only if $B$ is a trivial cyclic left brace. In this case, $\alpha$ is given by $\alpha(a):=-a$ for all $a\in B$.
\end{lemma}

\begin{proof}
If $B$ is a trivial cyclic left brace, $\alpha$ is clearly one of its automorphisms. Conversely, suppose that $B=C(p,k,t)$ for some $k,t$. If $\alpha$ is an automorphism of order $2$ then in particular is an element of $Aut(B,+)$, hence it must be the function given by $\alpha(a):=-a$ for all $a\in B$. Moreover, we have that $-(a\circ b)=(-a)\circ (-b)$ for all $a,b\in B$, therefore $-a-b-p^t ab=-a-b+p^tab$ for all $a,b\in B$. If we set $a=b=1$, since $p$ is odd we obtain that $p^t$ is equal to $0$ module $p^k$, hence the thesis.
\end{proof}

\begin{lemma}\label{brdihedral}
Let $B$ be a cyclic left brace of size $2n$, for some odd number $n$. Moreover, suppose that $(B,\circ)$ is isomorphic to a dihedral group. Then, $B$ is the semidirect product of the trivial cyclic left brace of size $2$ and the trivial cyclic left brace of size $n$.
\end{lemma}

\begin{proof}
By \cite[Theorem 2 and Corollary 2]{rump2019classification}, $B$ is the semidirect product of the trivial cyclic left brace of size $2$ and a cyclic left brace of size $n$. Since $\alpha(B_p)=B_p$ for every divisor $p$ of $n$, every Sylow subgroup $B_p$ of $(B,+)$ and $\alpha\in Aut(B,+,\circ)$, by \cref{autcyc} the left brace of size $n$ must be trivial.
\end{proof}

\noindent Now we are able to give the first main result of the section.

\begin{theor}\label{soldihe3.3}
Let $n$ be an odd number. Then, there exist a unique indecomposable solution such that  the permutation group is isomorphic to the dihedral group of order $2n$. Moreover, such a solution is uniconnected.
\end{theor}

\begin{proof}
By \cref{cosmod}, every solution with dihedral permutation group can be constructed starting from a left brace with dihedral multiplicative group. 
By \cref{brdihedral}, there exist a unique left brace $B$ having order $2n$ and with dihedral multiplicative group and it is the semidirect product of the trivial cyclic left brace of size $2$, which we call $B_2$, and the trivial cyclic left brace of size $n$, which we call $B_n$.
 Therefore, it is sufficient to show that $B=B_n\rtimes B_2$ provide a unique indecomposable solution. At first, by \cref{uniconndihe} it follows that $B$ only provides uniconnected solutions. Therefore, if $(X_1,r_1)$ and $(X_2,r_2)$ are indecomposable solutions constructed starting, respectively, by means of $a_1=(a_n,a_2)$ and $a_1'=(a_n',a_2')$ (as in by \cref{cosmod}), by \cite[Lemma 3.2]{castelli2021classification} we have that necessarily $a_2=a_2'$. Moreover, since again by \cite[Lemma 3.2]{castelli2021classification} $a_n$ and $a_n'$ additively generate $B_n$, we have that there exist $\mu\in Aut(B_n,+)$ such that $\mu(a_n)=a_n'$. Therefore, since $Aut(B_n,+)$ canonically embeds into $Aut(B,+,\circ)$, by \cref{bachi} the thesis follows.
\end{proof}

\noindent As a consequence of the previous theorem, we have the following result, which is a limitation on the possible orders of the set $X$.

\begin{prop}
    Let $(X,r)$ be an indecomposable involutive solution with dihedral permutation group. Then, the size of $X$ is an even number.
\end{prop}

\begin{proof}
Let $n$ be the size of $X$ and suppose that $n$ is odd. By \cref{soldihe}, $\mathcal{G}(X,r)$ must have size $2n$ and, by the previous theorem, this implies that $X$ has size $2n$, a contradiction.
\end{proof}

\noindent Before giving the second main result of this section, we have to recall some basic facts on left braces.

\begin{rem}\label{basiccyclic}
    By results obtained in \cite[Sections 3-4]{rump2019classification}, we have that for every group $G$ having square-free order, there exist a unique left brace $B$ with multiplicative group isomorphic to $G$. Moreover, every cyclic left braces having square-free order is a semidirect product of two trivial cyclic left braces with  coprime size. Therefore, if $B$ is such a left brace, we obtain that there exist two trivial left braces $B_1$ and $B_2$ such that $B$, as left brace, is isomorphic to the direct product $(B_1\rtimes B_2)\times Z(B,\circ)$, where $ Z(B,\circ)$ is the trivial left brace on the center of $(B,\circ)$. These facts imply the following result.
\end{rem}





 \begin{cor}\label{unisqfree}
     Every indecomposable solution having a permutation group of square-free size is uniconnected.
 \end{cor}

 \begin{proof}
 The thesis follows by the previous remark, toegether with  \cref{cosmod} and \cref{solcyc}.
 \end{proof}

 \noindent Now, we are able to classify all the indecomposable solutions whose permutation group has  square-free size.

 \begin{theor}
Let $G$ be a finite group having square-free size and let $B$ be the unique left brace having $G$ as multiplicative group. If $B_1, B_2$ and $Z(B,\circ)$ are the left braces as in \cref{basiccyclic}, the elements of $B_2$ determine the isomorphism classes of the indecomposable solutions having $G$ as permutation group.  
 \end{theor}

 
 \begin{proof}
By \cref{unisqfree}, every indecomposable involutive solution with permutation group isomorphic to $G$ is uniconnected. By means of \cref{cosmod}, such a solutions having $G$ as permutation group can be constructed by $B$. Therefore, the statement can be showed as \cite[Corollary $6.2$]{castelli2021classification} using the left brace $B$.
 \end{proof}

 \noindent Therefore, to construct easily all the indecomposable solution whose permutation group $G$ has square-free size, it is sufficient to consider the left brace $B$ having $G$ as permutation group, a generator $a$ of the additive group $(B,+)$ and then we can use the formula given in \cref{cosmod} taking $K$ as the trivial permutation group. By the previous theorem, we can distinguish the isomorphism classes.

\section{Indecomposable solutions with minimal non-cyclic permutation group}

In this section, we study indecomposable involutive solutions with minimal non-cyclic group. In particular, we prove that these solutions always are uniconnected and we classify them.

\medskip

\noindent First of all, recall that a finite group $G$ is said to be a \emph{minimal non-cyclic group} if it is not a cyclic group and every proper subgroup is cyclic. These groups are completely classified by Miller and Moreno.

\begin{lemma}[See \cite{miller1903non}]\label{gruc}
A finite group $G$ is a minimal non-cyclic group if and only if G is one of the following groups:
\begin{itemize}
    \item[a)] $(\mathbb{Z}/p\mathbb{Z}\times \mathbb{Z}/p\mathbb{Z},+)$, where $p$ is a prime number;
    \item[b)] the quaternion group $Q_8$;
    \item[c)] $<a,b\; |\; a^{p}=b^{q^n}=1,b^{-1}ab=a^r>$, where $p,q$ are distinct prime numbers, $r\neq 1$, $r\equiv_{q} 1$ and $r^q\equiv_p 1$.
\end{itemize}
\end{lemma}

\begin{conv}
    From now on, the groups of the previous lemma will be called of \emph{type} a), b) and c).
\end{conv} 

\begin{theor}\label{uniconnmin}
Let $(X,r)$ be an indecomposable involutive solution such that $\mathcal{G}(X,r)$ is a minimal non-cyclic group. Then, $(X,r)$ is uniconnected.
\end{theor}

\begin{proof}
If $\mathcal{G}(X,r)$ is isomorphic to the group a) of \cref{gruc}, then it is abelian and hence $(X,r)$ clearly is uniconnected. If $\mathcal{G}(X,r)$ is isomorphic to the group b) of \cref{gruc}, then $(X,r)$ is uniconnected by \cref{solded}. Finally, suppose that $\mathcal{G}(X,r)$ is isomorphic to the group c) of \cref{gruc}. If $(q,n)\neq (2,2)$, then $(X,r)$ is uniconnected by \cref{solZ}, hence the unique case to consider is when $(q,n)$ is equal to $(2,2)$. In this case, we have that $\mathcal{G}(X,r)$, as a group, is isomorphic to $A_p\rtimes A_2$, where $A_p$ and $A_2$ are the left ideals of $\mathcal{G}(X,r)$ (regarded as permutation left brace) given by the Sylow subgroups of $(\mathcal{G}(X,r),+) $. Moreover, since $A_p$ has a trivial left brace structure, $\mathcal{G}(X,r) $, as left brace, is just the semidirect product of the left braces $A_p$ and $A_2$, where $(A_2,\circ)$ acts by inversion on $A_p$. By \cref{cosmod}, there exist an element $(a_2,b_2)\in A_p\rtimes A_2$ and a core-free subgroup $H$ contained in $Stab(a_2,b_2) $ such that $X$ can be identified with the left cosets $\mathcal{G}(X,r)/H$ and the $\lambda$-orbit of $(a_2,b_2) $ additively generates the whole permutation left brace $\mathcal{G}(X,r)$. Moreover, since the image of $\mathcal{G}(X,r)$ by $\lambda$ is isomorphic to a quotient of $A_2$, and hence is abelian, we obtain that $H$ stabilizes the whole $\lambda$-orbit of $(a_2,b_2)$, and this implies that $H\subseteq Soc(\mathcal{G}(X,r))$. By a standard calculation, we obtain that the socle of $\mathcal{G}(X,r)$ is an abelian subgroup with $|Soc(\mathcal{G}(X,r))|\in \{p,2p \} $, therefore it follows that $H$ is a characteristic subgroup of $Soc(\mathcal{G}(X,r)) $, hence it is normal in $\mathcal{G}(X,r) $. Since $H$ is core-free, necessarily $|H|=1$ and hence $|X|=|\mathcal{G}(X,r)|$.
\end{proof}

\noindent At first, we consider indecomposable involutive solutions with permutation group of type a) and b).

\begin{prop}
There is a unique indecomposable involutive solution $(X,r)$ with minimal non-cyclic permutation group $\mathcal{G}(X,r)$ of type a) (resp. b)).
\end{prop}

\begin{proof}
If $\mathcal{G}(X,r)$ is of type a) the thesis follows by \cite[Theorem 21]{capiru2020}. Now, by \cref{uniconnmin} $(X,r)$ is uniconnected, hence if $\mathcal{G}(X,r)$ is of type b) the thesis follows by an inspection of the solutions having size $8$ \cite{Ve15pack}.
\end{proof}

\noindent Therefore, by the previous proposition, we have to focus on indecomposable involutive solutions with permutation group of type c).\\
For our purpose, we have to determine left braces with multiplicative group isomorphic to groups of type c). If $B$ is such a left brace and $A_p,A_q$ are the left ideals given by the Sylow subgroups of $(B,+)$, we have that $(B,\circ)$ is the groups semidirect product $(A_p,\circ)\rtimes (A_q,\circ)$. Moreover, since $A_p$ has a trivial left brace structure, it follows that $B$, as left brace, is just the left braces semidirect product $(A_p,+,\circ)\rtimes (A_q,+,\circ)$. Hence, to find the left braces with multiplicative group isomorphic to groups of type c), we have to find the possible semidirect products of left braces having multiplicative group isomorphic to  $\mathbb{Z}/p\mathbb{Z}$ and $\mathbb{Z}/q^n\mathbb{Z}$.

\begin{prop}\label{minbrace1}
There exist a unique left brace $B$ such that $B$ is the semidirect product of the trivial left braces on $\mathbb{Z}/p\mathbb{Z}$ and $\mathbb{Z}/q^n\mathbb{Z}$ and $\mathbb{Z}/p\mathbb{Z}\rtimes \mathbb{Z}/q^n\mathbb{Z}$ is isomorphic to the group c). 
\end{prop}

\begin{proof}
By a standard verification, we have that two distinct semidirect product as in the statement are isomorphic.
\end{proof}

\begin{prop}\label{minbrace2}
Let $n$ be a natural number and $p,q$ be distinct prime numbers with $q<p$ such that $(q,n)\neq (2,2)$. Moreover, let $A_p$ be the trivial left brace of size $p$, $A_q$ a non-trivial left brace with cyclic multiplicative group of size $q^n$ and $\alpha,\alpha':(A_q,\circ)\longrightarrow Aut(A_p,+)$ be a homomorphisms such that $(A_p\rtimes_{\alpha} A_q,\circ)$ and $(A_p\rtimes_{\alpha'} A_q,\circ)$ are isomorphic to the group c). Then, the left braces $A_p\rtimes_{\alpha} A_q $ and $A_p\rtimes_{\alpha'} A_q$ are isomorphic if and only if  $\alpha=\alpha'$.
\end{prop}

\begin{proof}
At first, recall that by \cite[Theorem 2]{rump2007classification} $(A_q,+)$ is a cyclic group, therefore there exist $t\in \{1,...,n-1\}$ such that $A_q=C(q,t,n)$. 
If $\alpha=\alpha'$ the left braces $A_p\rtimes_{\alpha} A_q $ and $A_p\rtimes_{\alpha'} A_q$ clearly are isomorphic.\\
Conversely, suppose that $A_p\rtimes_{\alpha} A_q $ and $A_p\rtimes_{\alpha'} A_q$ are isomorphic as left braces and let $\psi$ be an isomorphism from $A_p\rtimes_{\alpha} A_q $ to $A_p\rtimes_{\alpha'} A_q$. Since $\psi$ must be an isomorphism between the additive groups, there exist two numbers $k_1,k_2$ such that $\psi(a_1,a_2)=(k_1 a_1,k_2 a_2)$ for all $(a_1,a_2)\in A_p\times  A_q$. Since $\psi$ also is an isomorphism between the multiplicative group, we have that
$$(k_1(a_1+\alpha_{a_2}(b_1)),k_2(a_2+b_2+q^ta_2 b_2))=(k_1 a_1+\alpha'_{k_2a_2}(k_1b_1),k_2a_2+k_2b_2+k_2^{2}q^ta_2b_2) $$
for all $(a_1,a_2),(b_1,b_2)\in A_p\rtimes_{\alpha} A_q$. Comparing the second components, we obtain that there exist $s,s'\in Soc(A_q)$ such that $k_2=1+s=1\circ s'$. Now, comparing the first components, we obtain that
$$k_1 \alpha_{a_2}(b_1)=k_1\alpha'_{k_2 a_2}(b_1) $$
for all $(a_1,a_2),(b_1,b_2)\in A_p\times A_q $. If we set $a_2=1$ (in $ A_q$) and $b_1=1$ (in $A_p$) we obtain that $k_1 \alpha_{1}(1)=k_1\alpha'_{k_2 }(1) $ that implies $\alpha_{1}(1)=\alpha'_{k_2 }(1) $. Since $k_2=1\circ s'$ with $s'\in Soc(A_q)$ and $Soc(A_q)\subseteq Ker(\alpha)$ (because $A_q$ is a non-trivial left brace and $Ker(\alpha) $ has index $q$ in the cyclic group $(A_q,\circ)$), we have that $\alpha_{1}(1)=\alpha'_{1 }(1) $, that implies $\alpha_1=\alpha'_1$. Since the multiplicative group of $A_q$ is generated by $1$, we obtain $\alpha=\alpha'$.
\end{proof}



\noindent Now, we can give the desired classification.

\begin{cor}\label{minind2}
Let $n$ be a natural number and $p,q$ be distinct prime numbers with $q<p$,
$A_p$ be the trivial left brace of size $p$, $A_q$ a trivial left brace with cyclic multiplicative group of size $q^n$ and $\alpha:(A_q,\circ)\longrightarrow Aut(A_p,+)$ be a homomorphism such that the multiplicative group of the left braces semidirect product $A_p\rtimes_{\alpha} A_q$ is isomorphic to the group c). Moreover, suppose that $(a_1,a_2),(b_1,b_2)\in A_p\rtimes_{\alpha} A_q$ are elements of two (not necessarily distinct) transitive cycle bases and set $(X_1,r_1)$ (resp. $(X_2,r_2)$) the indecomposable involutive solution obtained using $(a_1,a_2)$ (resp. $(b_1,b_2)$) and the trivial subgroup $K:=\{(0,0)\}$ of $(A_p\rtimes_{\alpha} A_q,\circ)$ in \cref{cosmod}. Then, $(X_1,r_1)$ and $(X_2,r_2)$ are isomorphic if and only if $a_2\equiv b_2$ $(mod\; q)$.
\end{cor}

\begin{proof}
Since $A_p$ and $A_q$ are cyclic left braces, then so $A_p\rtimes_{\alpha} A_q$ is. If $q\neq 2$, by \cite[Theorem 5.8]{castelli2021classification} $(X_1,r_1)$ and $(X_2,r_2)$ are isomorphic if and only if $a_2\equiv b_2$ $(mod\; q^z)$, where $z=min\{1,t\}$. Since $A_q$ is trivial we have that $t=n$, hence in this case the thesis follows. If $q=2$, even if we can not use \cite[Theorem 5.8]{castelli2021classification}, since $A_p\rtimes_{\alpha} A_q $ is a cyclic left brace the thesis can be showed in the same way.
\end{proof}

\begin{cor}\label{minind1}
Let $n$ be a natural number and $p,q$ be distinct prime numbers with $q<p$ such that $(q,n)\neq (2,2)$, $A_p$ be the trivial left brace of size $p$, $A_q$ a non-trivial left brace with cyclic multiplicative group of size $q^n$ and $\alpha:(A_q,\circ)\longrightarrow Aut(A_p,+)$ be a homomorphism such that the multiplicative group of the left braces semidirect product $A_p\rtimes_{\alpha} A_q$ is isomorphic to the group c). Moreover, suppose that  $(a_1,a_2),(b_1,b_2)\in A_p\rtimes_{\alpha} A_q$ are elements of two (not necessarily distinct) transitive cycle bases and set $(X_1,r_1)$ (resp. $(X_2,r_2)$) the indecomposable involutive solution obtained using $(a_1,a_2)$ (resp. $(b_1,b_2)$) and the trivial subgroup $K:=\{(0,0)\}$ of $(A_p\rtimes_{\alpha} A_q,\circ)$ in \cref{cosmod}. Then, $(X_1,r_1)$ and $(X_2,r_2)$ are isomorphic if and only if $a_2\equiv b_2$ $(mod\; q^z)$, where $z=min\{n-t,t\}$ and $q^t$ is the size of $Soc(A_q)$.
\end{cor}

\begin{proof}
By \cite[Theorem 2]{rump2019classification} $A_p$ and $A_q$ are cyclic left braces, then so $A_p\rtimes_{\alpha} A_q$ is. Since $A_q$ is not trivial and $Ker(\alpha)$ is a subgroup of $A_q$ having index $q$, we have that $Soc(A_q)\subseteq Ker(\alpha)$, therefore if $q\neq 2$ the thesis follows by \cite[Theorem 5.8]{castelli2021classification}. Even if the case $q=2$ is not covered by \cite[Theorem 5.8]{castelli2021classification}, since $A_p\rtimes_{\alpha} A_q$ is a cyclic left brace, the thesis can be showed in the same way.
\end{proof}


\noindent The remaining case is $|X|=4p$, that we consider below.

\begin{cor}\label{minind3}
Let $p$ be a prime number different from $2$. Then, are two indecomposable solutions $(X,r)$ such that $\mathcal{G}(X,r)$ is isomorphic to the minimal non-cyclic group of size $4p$.
\end{cor}

\begin{proof}
By \cref{uniconnmin}, $(X,r)$ is uniconnected. The permutation left brace $\mathcal{G}(X,r)$ is a semidirect product of the left brace $A_2$ and the left brace $A_p$. If $A_2$ is a trivial left brace, then by \cref{cosmod} and \cref{minind2} we have a unique indecomposable solution provided by $\mathcal{G}(X,r)$, using a generator $(a,b)$ of $(A_p,+)\times (A_2,+)$.
Now, suppose that $A_2$ is the unique non-trivial left brace having size $4$ and with cyclic multiplicative group. 
Hence, the semidirect product is completely determined by the unique homomorphism that sends a generator (which we call $1$) of $(A_2,\circ)$ to the unique homomorphism having order $2$ of $A_p$. Now, let $(a,b)\in A_p\times A_2$ an element of a transitive cycle base of $A$ such that $(X,r)$ is isomorphic to the indecomposable solution obtained by $(a,b)$ as in \cref{cosmod}. Then, $b$ must be one of the two generators of $ (A_2,\circ)$. Since the function from $A_p\rtimes A_2$ to itself given by $\psi(x,y)=(\beta(x),y)$ for some $\beta\in Aut(A_p,+)$ is an automorphism of the left brace $A_p\rtimes A_2$, by \cref{bachi} the isomorphism class of $(X,r)$ does not depend on $a$. Moreover, $\lambda_{(0,1)}(a,b)$ belongs to $A_p\times \{b+2\}$, therefore we have that by \cref{bachi} the isomorphism class of $X$ does not depend on $b$. Hence, we showed that in this case also we provide a unique solution.
\end{proof}

\noindent Therefore, to find \emph{all} the indecomposable solutions with minimal non cyclic group of type c), we can apply the following steps:
\begin{itemize}
    \item[1)] Use \cref{minbrace1} and \cref{minbrace2} to construct, up to isomorphisms, all the suitable left braces.
    \item[2)] Use \cref{minind2}, \cref{minind1} and \cref{minind3} to provide all the non-isomorphic indecomposable solutions having a fixed left brace of the step 1).
\end{itemize}

\noindent In the last theorem, we summarize all the results of this section.

\begin{theor}
    Let $p,q$ distinct prime numbers, $n$ a natural number and suppose that $G$ is a minimal non-cyclic group having size $pq^n$ and of type c). Then, every indecomposable solution with permutation group isomorphic to $G$ is uniconnected. Except the case $q^n=4$, every indecomposable solutions $(X,r)$ with $\mathcal{G}(X,r)\cong G$ are obtained by means of the left braces provided in \cref{minbrace1} and \cref{minbrace2}, using \cref{cosmod}. In particular, we have:
    \begin{itemize}
        \item[-] $1$ solution, if $q^n=4$ and $\mathcal{G}(X,r)\cong C(p,1,1)\rtimes C(2,2,2) $;
        \item[-] $1$ solution, if $q^n=4$ and $\mathcal{G}(X,r)\cong C(p,1,1)\rtimes C(2,2,1) $;
        \item[-] $q-1$ solutions, if $q^n\neq 4$ and $\mathcal{G}(X,r)\cong C(p,1,1)\rtimes C(q,n,n) $;
        \item[-] $q^z-1$ solutions, if $q^n\neq 4$ and $\mathcal{G}(X,r)\cong C(p,1,1)\rtimes C(q,n,t) $ for some $t\in \{1,...,n-1 \}$, with $z=min\{n-t,t\}$.
    \end{itemize}
\end{theor}

\begin{proof}
    It follows by all the previous results of this section.
\end{proof}

\section{Dehornoy's class and indecomposable involutive solutions}

In this section, which is independent from the previous ones, we focus on the Dehornoy's class of involutive solutions (not necessarily indecomposable). In \cite{feingesicht2023dehornoy} the following conjectures were stated.


\begin{conj}[Conjecture 3.4, \cite{feingesicht2023dehornoy}]\label{conge1}
Let $(X,r)$ be an indecomposable involutive solution. Then, the Dehornoy's class is bounded by the size of $X$.
\end{conj}

\begin{conj}[Conjecture 3.6, \cite{feingesicht2023dehornoy}]\label{conge2}
Let $(X,r)$ be an involutive solution having size $n$. Then, the Dehornoy's class is bounded by $a_n$, where 
$$a_n:=max\{\prod_{i=1}^k n_i|k\in \mathbb{N},\quad 1\leq n_1<...< n_k,\quad  n_1+...+n_k=n \}.$$
Moreover, the bound $a_n$ is minimal.
\end{conj}

\noindent Some evidences to these conjectures, by basic examples and computer calculations, were given in \cite[Sections 3-4]{feingesicht2023dehornoy}. The aim of this section is to provide new estimations of the Dehornoy's class for several families of involutive solutions. These estimations will be used to prove \cref{conge1} and the bound part of \cref{conge2} in several cases.

\medskip

\begin{rem}\label{remark1}
 Using the same terminology of \cref{costruz}, recall that every involutive solution $(X,r)$ can be constructed by a quadruple $\{B,I,\mathcal{A},\mathcal{K} \} $. From now on, we will indicate this such a solution by $(X_{B,I,\mathcal{A},\mathcal{K}},r)$. Again by \cref{costruz}, it is not restrictive supposing that every solutions has the form $(X_{B,I,\mathcal{A},\mathcal{K}},r)$ for some suitable quadruple $\{B,I,\mathcal{A},\mathcal{K} \} $. Moreover, for indecomposable solutions, by \cite[Theorem 3]{rump2020classifi} we can choose a quadruple with $|I|=1$.
     
 \end{rem}

 \begin{conv}
     From now on, given an involutive solution $(X_{B,I,\mathcal{A},\mathcal{K}},r)$ we will indicate by $d$ its Dehornoy's class.
 \end{conv}

\subsection{Dehornoy's class and left braces}

\noindent This section is devoted to explain the precise link between an involutive solution $(X_{B,I,\mathcal{A},\mathcal{K}},r)$ and its permutation left brace $B$. Before giving the first result, we need the following lemma.

\begin{lemma}\label{ord}
For an involutive solution $(X_{B,I,\mathcal{A},\mathcal{K}},r)$ the equality
$$\Omega_n(x\circ K_{i_1,j_1},...,x\circ K_{i_1,j_1},y\circ K_{i_2,j_2})=(((n-1)\lambda_x(a_{i_1}))^-)\circ y \circ K_{i_2,j_2}$$
holds for all $n\in \mathbb{N}$ and $x\circ K_{i_1,j_1}, y\circ K_{i_2,j_2}\in X_{B,I,\mathcal{A},\mathcal{K}} $
\end{lemma}

\begin{proof}
We show the thesis by induction on $n$. If $n\in \{1,2\}$, the thesis is clear. Now, suppose the thesis for $n$ and we show the thesis for $n+1$. Then

\begin{equation*}
\begin{split}
& \Omega_{n+1}(x\circ K_{i_1,j_1},...,x\circ K_{i_1,j_1},y\circ K_{i_2,j_2})  = \noindent \\
& =\Omega_{n}(x\circ K_{i_1,j_1},...,x\circ K_{i_1,j_1})\cdotp  \Omega_{n}(x\circ K_{i_1,j_1},...,y\circ K_{i_2,j_2}) \\
& = (((n-1)\lambda_x(a_{i_1}))^-)\circ x \circ K_{i_1,j_1})\cdotp (((n-1)\lambda_x(a_{i_1}))^-)\circ y \circ K_{i_2,j_2})\\
& = (\lambda_{((n-1)\lambda_x(a_{i_1}))^-)\circ x}(a_{i_1}))^{-}\circ (((n-1)\lambda_x(a_{i_1}))^-)\circ y \circ K_{i_2,j_2}) \\
& = ( ((n-1)\lambda_x(a_{i_1})) \circ \lambda_{((n-1)\lambda_x(a_{i_1}))^-)\circ x}(a_{i_1})  ))^{-} \circ y \circ K_{i_2,j_2}) \\
& = ( ((n-1)\lambda_x(a_{i_1})) + \lambda_{x}(a_{i_1})  ))^{-} \circ y \circ K_{i_2,j_2}) \\
& =  (n\lambda_x(a_{i_1})))^{-} \circ y \circ K_{i_2,j_2} \\
\end{split}
\end{equation*}
for all $x\circ K_{i_1,j_1}, y\circ K_{i_2,j_2}\in X_{B,I,\mathcal{A},\mathcal{K}}$, where in the first equality we used the inductive hypothesis and in the fifth equality we used the left brace equality $z\circ \lambda_{z^-\circ w}(v)=z+\lambda_w(v)$.
\end{proof}

\noindent In \cite[Theorem G]{lebed2022involutive}, it was shown that the class of an involutive solution $(X,r)$ is the least common multiple of the orders of the permutations $\sigma_x$ in the additive group of $(\mathcal{G}(X,r),+,\circ)$. Here, we recover this result showing the precise link between the orders of the $\sigma_x$ and the elements of the permutation left brace.

\begin{prop}\label{lcm}
The Dehornoy's class of $(X_{B,I,\mathcal{A},\mathcal{K}},r)$ is equal to the least common multiple of the additive orders of the $\{a_i\}_{i\in I}$ in $(B,+)$. 
\end{prop}

\begin{proof}
    By \cref{ord}, we have that $(X_{B,I,\mathcal{A},\mathcal{K}},r)$ is of class $m$, for a natural number $m$, if and only if $(((m)\lambda_x(a_{i_1}))^-)\circ y \circ K_{i_2,j_2}=y \circ K_{i_2,j_2} $ for all $x,y\in B$, $i_1,i_2\in I$, $j_1\in J_{i_1}$ and $j_2\in J_{i_2}$, and this is equivalent to have $m\lambda_x(a_{i_1})\in  \bigcap_{i,j} core(K_{i,j})$. Since $o(a_i)_+=o(\lambda_x(a_i))_+$ for all $x\in B$ and, by \cref{costruz}, we have $|\bigcap_{i,j} core(K_{i,j})|=1$, it follows that $m$ must be equal to the least common multiple of the orders of the $\{a_i\}_{a_i\in I}$ in $(B,+)$. 
\end{proof}

\noindent Given a finite group $(G,\circ)$, recall that the \emph{exponent} of $G$, which we indicate by $exp(G,\circ)$, is the natural number given by the least common multiple of their elements.

\begin{cor}\label{exp}
    The Dehornoy's class $d$ of $(X_{B,I,\mathcal{A},\mathcal{K}},r)$ is equal to $exp(B,+)$. In particular, $d$ divides $|B|$.
\end{cor}

\begin{proof}
    By \cref{lcm}, we have that $d$ divides $exp(B,+)$. Since $\bigcup_{i\in I} B_i $ additively generates $B$, we have $o(z)_+$ divides $d$, for every $z\in B$, hence $exp(B,+)$ divides $d$ and the equality $d=exp(B,+)$ follows. The second part follows since exponent divides order in a finite group.
\end{proof}

\begin{cor}\label{exp2}
    The Dehornoy's class $d$ of $(X_{B,I,\mathcal{A},\mathcal{K}},r)$ is equal to the maximum of the set $\{o(x)_+|x\in B\}$.
\end{cor}

\begin{proof}
    Since in every finite abelian group there exist an element such that its order is equal to the exponent, the statement follows by \cref{exp}.
\end{proof}

\noindent The previous results state that if we have suitable informations on the additive group of a left brace $(B, +,\circ)$, then we can estimate the Dehornoy’s class of the involutive solutions provided by $B$.




\begin{defin}
    If $n\in \mathbb{N}$, we indicate by $g(n)$ the maximum order of a permutation in $Sym(n)$.
\end{defin}

\noindent By a standard calculation, we have that $g(n)\leq a_n$ for all $n\in \mathbb{N}$. In \cite[Proposition 3.11] {feingesicht2023dehornoy} the author showed that the bound part of \cref{conge2} follows if $(B,\circ)$ is abelian and $(X_{B,I,\mathcal{A},\mathcal{K}},r)$ is square-free. In particular, it was shown that, in this case, the Dehornoy's class is bounded by $g(n)$. If $(B,\circ)$ is cyclic, the square-freeness can be dropped.

\begin{prop}
Suppose that $(X_{B,I,\mathcal{A},\mathcal{K}},r)$ has cyclic permutation group (in other words, $(B,\circ)$ is a cyclic group) and let $n$ be such that $n=|X_{B,I,\mathcal{A},\mathcal{K}} | $. Then the Dehornoy's class of  
 $(X_{B,I,\mathcal{A},\mathcal{K}},r)$ is bounded by $g(n)$. 
\end{prop}

\begin{proof}
    We have that $(B,\circ)$ is generated by a permutation of $Sym(n) $, hence the size of $B$ is bounded by $g(n)$. Moreover, by \cref{exp} $d$ divides $|B|$, therefore the thesis follows.
\end{proof}

\noindent The following easy example shows that the bound obtained in the previous proposition is the best possible, if $(X_{B,I,\mathcal{A},\mathcal{K}},r)$ has cyclic permutation group.

\begin{ex}
    Let $X:=\{1,...,n\}$ and $\gamma$ a permutation of $Sym(X)$ having order $s$. Then, the map $\sigma_x$ from $X$ to itself given by $\sigma_x(y):=\gamma(y)$ for all $x,y\in X$ give rise to an involutive solution $(X,r)$ having cyclic permutation group. Moreover, since $\Omega_n(x_1,...,x_n):=\gamma^{n-1}(x_n)$ for all $x_1,...,x_n\in X$, we have that $(X,r)$ is of class $s$.
\end{ex}

\noindent It is well-known that every finite abelian group is isomorphic to a direct product of cyclic groups having prime-power size. In the following result we show that the bound part of \cref{conge2} follows for a large family of involutive solutions with abelian permutation group (not necessarily square-free).

\begin{prop}
Suppose that $(B,\circ)$ is isomorphic to $\mathbb{Z}/p_1^{\alpha_1}\mathbb{Z}\times ... \times  \mathbb{Z}/p_r^{\alpha_r}\mathbb{Z}$, where $p_1,...,p_r$ are prime number not necessarily distinct. Moreover, suppose that $\alpha_i \neq \alpha_j$ whenever $p_i= p_j$. Then the class of   $(X_{B,I,\mathcal{A},\mathcal{K}},r)$ is bounded by $a_n$.
\end{prop}

\begin{proof}
    Let $m:=p_1^{\alpha_1}+...+p_r^{\alpha_r} $, then we have that 
    $$d=exp(B,+)\leq \prod_{i=1}^r p_i^{\alpha_i} \leq a_m $$
    where the last inequality follows since by hypothesis $\alpha_i \neq \alpha_j$ whenever $p_i= p_j$. By \cite[Theorem 4]{ore1939contributions}, we have that $n=|X_{B,I,\mathcal{A},\mathcal{K}} |\geq m $ and hence $d\leq a_m\leq a_n$. 
\end{proof}

\noindent By results of \cite[Sections 3-4]{feingesicht2023dehornoy}, we know that every involutive solution having size at most $8$ has class bounded by $a_n$ (where $n$ is the size of the solution). If $n\leq 8$, the unique value for which $g(n)<a_n$ is $n=6$. By an inspection of involutive solutions having size $6$, we know that all the ones with abelian permutation group have Dehornoy's class bounded by $g(n)$. Thus, it is natural to ask if this is true for every involutive solution with abelian permutation group.

\begin{conj}
    The Dehornoy's class of an involutive solution $(X,r)$ with abelian permutation group is bounded by $g(n)$.
\end{conj}

\noindent Other involutive solutions $(X_{B,I,\mathcal{A},\mathcal{K}},r)$ whose Dehornoy's class is bounded by $g(n)$, where $n=|X_{B,I,\mathcal{A},\mathcal{K}}| $, are provided by left braces $B$ with $\lambda_a(a)=a$ for all $a\in B$.

\begin{prop}
Suppose that $(X_{B,I,\mathcal{A},\mathcal{K}},r)$ is such that $\lambda_a(a)=a$ for all $a\in B$, and let $n$ be such that $n=|X_{B,I,\mathcal{A},\mathcal{K}} | $. Then the Dehornoy's class of  
 $(X_{B,I,\mathcal{A},\mathcal{K}},r)$ is bounded by $g(n)$. 
\end{prop}

\begin{proof}
  By induction on $m$, one can show that $mx=x^{\circ m}$ for all $m\in \mathbb{N}$ and $x\in B$, hence $o(x)_+=o(x)_\circ$ and therefore by \cref{exp} $d=exp(B,+)=exp(B,\circ )$. Since $(B,+)$ is abelian, there exist $z\in B$ such that $d=o(z)_+=o(z)_\circ$, therefore $d\leq g(n)$.
\end{proof}

In \cite{dehornoy2015set}, it was showed that $d\leq (n^2)!$. In \cite[Proposition 3.10]{feingesicht2023dehornoy} this bound was improved showing that that $d\leq n!$. Using a theorem due by Dixon, in the following result we furtherly improve this estimation.

\begin{prop}
    Let $n:=|X_{B,I,\mathcal{A},\mathcal{K}}|$. Then the Dehornoy's class of $(X_{B,I,\mathcal{A},\mathcal{K}},r)$ is bounded by ${24}^{\frac{n-1}{3}}$.
\end{prop}

\begin{proof}
    By \cref{exp}, the class is bounded by $|B|$. Since by \cite[Theorem 2.15]{etingof1998set} $(B,\circ)$ must be a solvable group, by \cite[Theorem 3]{dixon1967fitting} $|B|$ is bounded by ${24}^{\frac{n-1}{3}}$ and hence the thesis follows.
\end{proof}




\noindent We conclude the subsection giving the computation of the Dehornoy's class for a large class of involutive solutions with dihedral permutation group.

\begin{prop}\label{dihedral}
Suppose that $(X_{B,I,\mathcal{A},\mathcal{K}},r)$ is such that $B$ is isomorphic to a dihedral group of size $2^n\cdotp m $ for $n>3$ and $m$ an odd number. Then, $d$ is equal to $2^{n}\cdotp m$ if $B$ is a cyclic left brace and it is equal to $2^{n-1}\cdotp m$ otherwise. 
\end{prop}

\begin{proof}
    If $(B,+)$ is a cyclic group, then the thesis follows by \cref{exp2}. If $(B,+)$ is not cyclic, since $(B,+)$ is isomorphic to the direct product of their $p$-Sylow subgroups $(B_p,+)$ and $(B_p,\circ)$ is a cyclic subgroup of $(B,\circ)$ for all $p\neq 2$, by \cite[Proposition 5.4]{bachiller2016solutions} it follows that there exist an element $b_m$ such that $o(b_m)_+=m$. By \cite[Theorem 1]{rump2020classifi}, $(B_2,+)$ is isomorphic to $\mathbb{Z}/2\mathbb{Z}\times \mathbb{Z}/2^{n-1}\mathbb{Z} $, hence if $b_2$ is such that $o(b_2)_+=2^{n-1}$, we obtain that $o(b_2+b_m)_+=2^{n-1}\cdotp m$. Since $(B,+)$ is not cyclic, $2^{n-1}\cdotp m$ is the maximum order in $(B,+)$, hence the thesis follows by \cref{exp2}.
\end{proof}

\subsection{Dehornoy's class of indecomposable involutive solutions}

In this subsection, we use the results previously obtained to give some estimations of several families of indecomposable involutive solutions. In particular, we will show several istances for which \cref{conge1} follows.

\begin{cor}\label{cyclicuni}
If $(X_{B,I,\mathcal{A},\mathcal{K}},r)$ is uniconnected, then $d\leq |X|$.
\end{cor}

\begin{proof}
    Since $|X_{B,I,\mathcal{A},\mathcal{K}}|=|B|$, the thesis follows by \cref{exp}.
\end{proof}





\begin{rem}
    When $(B,\circ)$ is abelian, by \cite[Proposition $3$]{rump2007braces} $B$ is isomorphic, as left braces, to the direct product of the $p$-Sylow subgroups of the additive group. In this case the computation of $d$ can be reduced to the one having $B$ a prime-power number of elements.
\end{rem} 

\begin{prop}
If $(X_{B,I,\mathcal{A},\mathcal{K}},r)$ is indecomposable and $(B,\circ)$ is a cyclic left brace of prime power order $p^k$, then $d=|X_{B,I,\mathcal{A},\mathcal{K}}|$ if $(|X_{B,I,\mathcal{A},\mathcal{K}}|,mpl(X_{B,I,\mathcal{A},\mathcal{K}},r))\neq (4,2)$ and it is $2$ otherwise.
\end{prop}

\begin{proof}
    It follows by \cref{exp2} and \cite[Proposition 5.4]{bachiller2016solutions}.
\end{proof}

\begin{prop}
Suppose that $(X_{B,I,\mathcal{A},\mathcal{K}},r)$ is indecomposable and has square-free order. Then, $d$  is equal to $|X_{B,I,\mathcal{A},\mathcal{K}}|$.
\end{prop}

\begin{proof}
    By \cite[Theorem 4.1]{cedo2022indecomposable}, $p$ divides $|B|$ if and only if $p$ divides $|X_{B,I,\mathcal{A},\mathcal{K}} |$ and moreover every $p$-Sylow subgroup of the additive group of $B$ is elementary abelian, hence the maximum order in $(B,+)$ is equal to  $|X_{B,I,\mathcal{A},\mathcal{K}} |$. Therefore, the thesis follows by \cref{exp2}.
\end{proof}

\begin{cor}
Suppose that $(X_{B,I,\mathcal{A},\mathcal{K}},r)$ is such that $(B,\circ)$ is isomorphic to a generalized quaternion $2$-group of size $2^n$, for $n>4$. Then, $d$ is equal to $2^n $ if $B$ is cyclic and it is equal to $2^{n-1}$ otherwise. If $(X_{B,I,\mathcal{A},\mathcal{K}},r)$ is indecomposable, $d\leq |X|$.
\end{cor}

\begin{proof}
    By \cite[Theorem 1]{rump2020classifi}, $(B,+)$ is cyclic or isomorphic to $\mathbb{Z}/2\mathbb{Z}\times \mathbb{Z}/2^{n-1}\mathbb{Z}$. Therefore the thesis follows by \cref{exp2}.
\end{proof}

\begin{prop}
Suppose that $(X_{B,I,\mathcal{A},\mathcal{K}},r)$ is indecomposable and $(B,\circ)$ is isomorphic to a generalized dihedral group of size $2m$. Then, $d\leq |X_{B,I,\mathcal{A},\mathcal{K}}|$.
\end{prop}

\begin{proof}
    If $(X_{B,I,\mathcal{A},\mathcal{K}},r) $ is uniconnected the thesis follows by \cref{cyclicuni}. Otherwise, by \cref{soldihe} $(X_{B,I,\mathcal{A},\mathcal{K}},r)$ has size $m$. In this way, by \cref{solcyc} $(B,+)$ can not be cyclic, therefore by \cref{exp2} $d$ can not be $2m$, hence $d<2m$. Since $d$ must divide $2m$, $d\leq m= |X_{B,I,\mathcal{A},\mathcal{K}} |$.
\end{proof}

\begin{cor}
Suppose that $(X_{B,I,\mathcal{A},\mathcal{K}},r)$ is indecomposable and $(B,\circ)$ is isomorphic to a dihedral group of size $2^n\cdotp m $ for $n>3$ and $m$ an odd number. Then, $d$ is equal to $2^{n}\cdotp m$ if $B$ is cyclic and it is equal to $2^{n-1}\cdotp m$ otherwise. In every case, $d\leq |X_{B,I,\mathcal{A},\mathcal{K}}|$.
\end{cor}

\begin{proof}
    The first part follows by \cref{dihedral}. The last part follows by \cref{dihedral} and \cref{soldihe}.
\end{proof}

\noindent We highlight that the previous corollary can not be extended when $B$ is a left brace with $(B,\circ)$ a dihedral group of size $2^2 m$ or $2^3 m$. Indeed, inspecting by \cite{Ve15pack} the indecomposable solutions of size $\leq 8$, we find one indecomposable involutive solution having size $4$, with a dihedral permutation group of size $2^3$ and Dehornoy class $2$. 





\bibliographystyle{elsart-num-sort}
\bibliography{Bibliography}

\end{document}